\documentclass{amsart}

\usepackage{amssymb, amscd, amsthm, mathrsfs}

\newtheorem{Thm}{Theorem}[section]
\newtheorem{Prop}[Thm]{Proposition}
\newtheorem{Cor}[Thm]{Corollary}

\theoremstyle{remark}

\numberwithin{equation}{section}

\newcommand{\Order}{\mathcal{O}}
\newcommand{\into}{\hookrightarrow}
\newcommand{\onto}{\twoheadrightarrow}
\newcommand{\isomto}{\overset{\sim}{\to}}

\newcommand{\compose}{\mathrel{\circ}}
\newcommand{\tensor}{\mathbin{\otimes}}
\newcommand{\ideal}[1]{\mathfrak{#1}}

\newcommand{\Z}{\mathbb{Z}}

\newcommand{\F}{\mathbb{F}}
\newcommand{\Affine}{\mathbb{A}}
\newcommand{\et}{\mathrm{et}}
\newcommand{\zar}{\mathrm{zar}}

\newcommand{\Tr}{\mathrm{Tr}}
\newcommand{\Gm}{\mathbf{G}_{m}}
\newcommand{\Ga}{\mathbf{G}_{a}}
\newcommand{\Et}{\mathrm{Et}}
\newcommand{\RP}{\mathrm{RP}}
\newcommand{\FRP}{\mathrm{FRP}}
\newcommand{\RPS}{\mathrm{RPS}}
\newcommand{\RPSZ}{\mathrm{RPSZ}}
\newcommand{\gr}{\mathrm{gr}}
\newcommand{\formal}[1]{\mathop{\text{``$#1$''}}}
\newcommand{\Sch}{\mathrm{Sch}}

\DeclareMathOperator{\Hom}{Hom}

\DeclareMathOperator{\Ext}{Ext}
\DeclareMathOperator{\Spec}{Spec}

\DeclareMathOperator{\sheafhom}{\mathscr{H}\mkern-5mu\mathit{om}}

\title{Duality theories for $p$-primary \'etale cohomology III}
\author{Kazuya Kato}
\address{
	Department of Mathematics, University of Chicago,
	5734 S.\ University Ave,
	Chicago, IL 60637, USA
}
\email{kkato@math.uchicago.edu}
\author{Takashi Suzuki}
\thanks{The second named author is a Research Fellow of Japan Society for the Promotion of Science}
\address{
	Department of Mathematics, Chuo University,
	1-13-27 Kasuga, Bunkyo-ku, Tokyo 112-8551, JAPAN
}
\email{tsuzuki@gug.math.chuo-u.ac.jp}
\date{May 7, 2019}

\begin{document}

\begin{abstract}
	This paper is Part III of the series of work by the first named author
	on duality theories for $p$-primary \'etale cohomology,
	whose Parts I and II were published in 1986 and 1987, respectively.
	In this Part III, we study a duality for $p$-primary \'etale nearby cycles
	on smooth schemes over henselian discrete valuation rings
	of mixed characteristic $(0, p)$ whose residue field is not necessarily perfect.
\end{abstract}

\maketitle

\tableofcontents


\section{Introduction}
This is Part III of the series of work
\cite{Kat86} (Part I), \cite{Kat87} (Part II) by the first named author.
Part I gives a duality theory for $p$-primary \'etale sheaves
on smooth varieties in characteristic $p > 0$
with a relative theory for proper morphisms between them,
which is a relative version of Milne's duality theories \cite{Mil76}, \cite{Mil86}.
In that part, sheaves on a scheme $Y$ of characteristic $p$
are not considered over the usual small \'etale site $Y_{\et}$,
but over a much bigger site $Y_{\RP}$,
which is the category of relatively perfect $Y$-schemes endowed with the \'etale topology.
Recall from \cite{Kat86}, \cite{Kat87} that a $Y$-scheme $Y'$ is said to be \emph{relatively perfect}
if its relative Frobenius morphism $Y' \to Y'^{(p)}$ over $Y$ is an isomorphism.
If the base field $k$ satisfies $[k : k^{p}] = p^{r_{0}}$ for some finite $r_{0} \ge 0$
and $Y$ is a smooth $k$-scheme purely of dimension $d$,
then one of the results \cite[Theorem 4.3]{Kat86} in this case says that
the dlog part $\nu_{n}(r) = W_{n} \Omega_{Y, \log}^{r}$ of the de Rham-Witt sheaf
viewed as a sheaf on $Y_{\RP}$ plays the role of a dualizing sheaf,
where $r = r_{0} + d$.
This theory is generalized, in Part II \cite{Kat86}, to singular varieties $Y$.
It instead uses the site $Y_{\FRP}$ of flat relatively perfect $Y$-schemes endowed with the \'etale topology
and constructs a certain dualizing complex $K_{n, Y}$ over $Y_{\FRP}$.

In this paper, as Part III, we study a mixed characteristic version of the duality theory of Part I.
More precisely, let $K$ be a henselian discrete valuation field of mixed characteristic $(0, p)$
with ring of integers $\Order_{K}$ and residue field $k$.
Assume that $[k : k^{p}] = p^{r_{0}}$ for some finite $r_{0} \ge 0$.
Let $X$ be a smooth $\Order_{K}$-scheme of relative dimension $d$
and
	\[
			U
		\overset{j}{\into}
			X
		\overset{i}{\hookleftarrow}
			Y
	\]
the inclusions of the generic fiber $U$ and the special fiber $Y$.
Set $r = r_{0} + d$.
We say that a $Y$-scheme is \emph{relatively perfectly smooth}%
\footnote{
	The terminology ``perfectly smooth'' without ``relatively'' for morphisms between perfect schemes
	is introduced in \cite[Definition A.25]{Zhu17}.
	In \cite[Footnote 20 to Definition A.13]{Zhu17},
	Zhu attributes this type of usage of ``perfectly'' to Brian Conrad.
}
if it is Zariski locally isomorphic to the relative perfection (\cite[Definition 1.8]{Kat86}) of a smooth $Y$-scheme.
Let $Y_{\RPS}$ be the category of relatively perfectly smooth $Y$-schemes.
Let $X_{\RPS}$ be the category of $X$-schemes flat over $\Order_{K}$
whose special fibers are relatively perfectly smooth over $Y$.
Endow these categories with the \'etale topology.
Let $U_{\Et}$ be the big \'etale site of $U$.
The base change functors define morphisms of sites
	\[
			U_{\Et}
		\overset{j}{\to}
			X_{\RPS}
		\overset{i}{\gets}
			Y_{\RPS}.
	\]
Our duality is about the nearby cycle functor $R \Psi = i^{\ast} R j_{\ast}$ in this setting.
For an integer $s$,
denote the $s$-th Tate twist of the \'etale sheaf $\Lambda_{n} = \Z / p^{n} \Z$ on $U$
by $\Lambda_{n}(s)$.
In the derived category of sheaves of $\Lambda_{n}$-modules on $Y_{\RPS}$,
the derived tensor product is denoted by $\tensor^{L}$
and the derived sheaf-Hom functor by $R \sheafhom_{Y_{\RPS}}$.
The main theorem of this paper, which is Theorem \ref{thm: main}, is the following statement.

\begin{Thm} \label{thm: mainthm in Introduction}
	For any pair of integers $s, t$ with $s + t = r + 1$,
	there exists a canonical morphism
		\[
				R \Psi \Lambda_{n}(s)[s]
			\tensor^{L}
				R \Psi \Lambda_{n}(t)[t]
			\to
				\nu_{n}(r)
		\]
	that is a perfect duality:
		\[
				R \Psi \Lambda_{n}(s)[s]
			\isomto
				R \sheafhom_{Y_{\RPS}} \bigl(
					R \Psi \Lambda_{n}(t)[t],
					\nu_{n}(r)
				\bigr).
		\]
\end{Thm}

For example, if $X = \Spec \Order_{K}$ and $k$ is algebraically closed,
then this theorem, on $k$-valued points, gives an exact sequence
	\[
			0
		\to
			\Ext_{k_{\RPS}}^{1}(G_{n}, \Z / p^{n} \Z)
		\to
			H^{1}(K, \Z / p^{n} \Z)
		\to
			\Hom(\mu_{p^{n}}(K), \Z / p^{n} \Z)
		\to
			0,
	\]
where $G_{n} = R^{1} \Psi \Lambda_{n}(1)$ is the group $K^{\times} / (K^{\times})^{p^{n}}$
equipped with a structure of the perfection of a unipotent algebraic group over $k$
and $\Ext_{k_{\RPS}}^{1}$ takes this structure into account
(and $\mu_{p^{n}}(K)$ is the finite abstract group of $p^{n}$-th roots of unity in $K$).
This recovers the $p$-primary part of the mixed characteristic case of
Serre's local class field theory \cite{Ser61} for $K$.
Hazewinkel's generalization \cite[Appendice]{DG70} of Serre's theory with arbitrary perfect residue field $k$
is also contained in this theorem.
This sheaf-theoretic formulation of Serre-Hazewinkel's theory is closely related to
the formulation using the ``rational \'etale site'' introduced in \cite{Suz13}.

For $X = \Spec \Order_{K}$ and not necessarily perfect $k$,
the theorem is a sheaf-theoretic version of class field theory
for local fields whose residue field is of arithmetic nature,
such as \cite{Kat80}, \cite{Par84} (for $k$ a higher local field) and \cite{Kat82} (for $k$ a global function field).
For general $X$ and $k$,
it is a $p$-primary version of a special case of the duality for prime-to-$p$ nearby cycles
by Gabber-Illusie \cite[Th\'eor\`eme 4.2]{Ill94}.

In the proof of the theorem,
we will use the computations of graded pieces of $p$-primary nearby cycles
by the first named author with Bloch \cite{BK86}.
Since the calculations in \cite{BK86} are for smooth schemes over $\Order_{K}$,
we limit ourselves with $X$ smooth over $\Order_{K}$
and work with the site $Y_{\RPS}$ of relatively perfectly smooth $Y$-schemes for simplicity.
Of course our duality theory should be extended to non-smooth (or at least semistable) $X$ and
more general constructible coefficients than $\Lambda_{n}(s)$,
so that it fully gives a mixed characteristic version of Part II \cite{Kat87} of the present series of work.
We will not pursue such an extension in this paper.

The first author is partially supported by NSF grant DMS 1601861.
The second author is supported by JSPS KAKENHI Grant Number JP18J00415.


\section{The relatively perfectly smooth site}
\label{sec: The relatively perfectly smooth site}

Let $k$ be a field of characteristic $p > 0$
such that $[k \colon k^{p}] = p^{r_{0}}$ for some finite $r_{0} \ge 0$.
Let $Y$ be a $k$-scheme.
Recall from \cite[Definition 1.1]{Kat86} that a $Y$-scheme $Y'$ is said to be relatively perfect
if the diagram
	\[
		\begin{CD}
			Y' @>>> Y' \\
			@VVV @VVV \\
			Y @>>> Y
		\end{CD}
	\]
is cartesian, where the horizontal morphisms are the absolute Frobenius morphisms
and the vertical morphisms are the structure morphisms.
This means that the relative Frobenius morphism $Y' \to Y'^{(p)}$ over $Y$ is an isomorphism.
Let $Y_{\RP}$ be the relatively perfect site of $Y$ defined in \cite[\S 2]{Kat86}.
It is the category of relatively perfect $Y$-schemes endowed with the \'etale topology.
Let $(\Sch / Y)$ be the category of all $Y$-schemes.
Assume that $Y$ is smooth over $k$.
Then the inclusion functor $Y_{\RP} \into (\Sch / Y)$
admits a right adjoint $(\Sch / Y) \to Y_{\RP}$
denoted by $Y' \mapsto Y'^{\RP}$ (\cite[Definition 1.8]{Kat86}),
and $Y'^{\RP}$ is called the relative perfection of $Y'$.
Let $n \ge 1$ be an integer.
Denote $\Lambda_{n} = \Z / p^{n} \Z$
and set $\Lambda = \Lambda_{1} = \Z / p \Z$.
For a site $S$, we denote the category of sheaves of $\Lambda_{n}$-modules on $S$ by $M(S, \Lambda_{n})$
and its derived category by $D(S, \Lambda_{n})$.
The ring $\Lambda_{n}$ viewed as a sheaf of rings on $S$ is denoted by
$(\Lambda_{n})_{S}$ or simply just $\Lambda_{n}$.
As in \cite[Definition 4.2.3]{Kat86},
we denote by $D_{0}(Y_{\RP}, \Lambda_{n})$
the triangulated subcategory of $D(Y_{\RP}, \Lambda_{n})$
generated by relative perfections of coherent sheaves on $Y$ locally free of finite rank
regarded as complexes of $\Lambda_{n}$-modules concentrated in degree zero.
As explained in \cite[\S 4]{Kat86},
if $Y$ is finite-dimensional,
the dlog part of the de Rham-Witt complex
$\nu_{n}(s) = W_{n} \Omega_{Y, \log}^{s}$ for any $s$
can be regarded as an object of $D_{0}(Y_{\RP}, \Lambda_{n})$
such that its section $\Gamma(Y', \nu_{n}(s))$ for any relatively perfect $Y$-scheme $Y'$
is given by the group $W_{n} \Omega_{Y', \log}^{s}$.
We set $\nu(s) = \nu_{1}(s) = \Omega_{Y, \log}^{s}$.

\begin{Thm}[{\cite[Theorem 4.3]{Kat86}}]
	Assume that $Y$ is smooth and purely of dimension $d$.
	Set $r = r_{0} + d$.
	Then the object $\nu_{n}(r)$ is a dualizing object for $D_{0}(Y_{\RP}, \Lambda_{n})$,
	namely the derived sheaf-Hom functor
		\[
				D_{Y_{\RP}}
			=
				R \sheafhom_{(\Lambda_{n})_{Y_{\RP}}}(\;\cdot\;, \nu_{n}(r))
		\]
	for $D(Y_{\RP}, \Lambda_{n})$ gives an auto-equivalence
	on $D_{0}(Y_{\RP}, \Lambda_{n})$ with inverse itself.
\end{Thm}

Let $Y$ be a smooth $k$-scheme.
We say that a $Y$-scheme is relatively perfectly smooth
if it is Zariski locally isomorphic to the relative perfection
of a smooth $Y$-scheme.
Let $Y_{\RPS}$ be the category of relatively perfectly smooth $Y$-schemes with $Y$-scheme morphisms.
Endow it with the \'etale topology.
The inclusion functor $Y_{\RPS} \into Y_{\RP}$ defines a morphism of topologies
in the sense of \cite[Definition 2.4.2]{Art62}.
It induces an (exact) pushforward functor $\alpha_{\ast}$
from $M(Y_{\RP}, \Lambda_{n})$ to $M(Y_{\RPS}, \Lambda_{n})$,
which has a left adjoint $\alpha^{\ast}$.
(Note that $\alpha^{\ast}$ is not necessarily exact
since the category of smooth $Y$-schemes is not closed under finite inverse limits.)
Note that $\alpha_{\ast}$ sends flask sheaves to flask sheaves
and induces Leray spectral sequences (\cite[\S 2.4]{Art62}).

\begin{Prop} \label{prop: from RP to RPS}
	Let $Y$ be a smooth $k$-scheme
	and $F \in D_{0}(Y_{\RP}, \Lambda_{n})$.
	Then the natural morphism
		\[
				\alpha_{\ast}
				R \sheafhom_{(\Lambda_{n})_{Y_{\RP}}}(F, G)
			\to
				R \sheafhom_{(\Lambda_{n})_{Y_{\RPS}}}(\alpha_{\ast} F, \alpha_{\ast} G)
		\]
	in $D(Y_{\RPS}, \Lambda_{n})$ for any $G \in D(Y_{\RP}, \Lambda_{n})$ is an isomorphism.
\end{Prop}

\begin{proof}
	It is enough to show the statement for the case
	where $F$ is the relative perfection of a coherent sheaf on $Y$ locally free of finite rank.
	Such a sheaf $F$ is representable by a relatively perfectly smooth $Y$-scheme.
	
	For a relatively perfectly smooth $Y$-scheme $Y'$,
	let $Y_{\RPS} / Y'$ be the localization of $Y_{\RPS}$ at $Y'$ (\cite[Definition 2.4.3]{Art62}),
	i.e., the category of $Y'$-schemes relatively perfectly smooth over $Y$ endowed with the \'etale topology.
	Taking $R \Gamma(Y'_{\RPS}, \; \cdot\;)$ for any relatively perfectly smooth $Y$-scheme $Y'$
	and using the Leray spectral sequence,
	we see that it is enough to show the invertibility of the morphism
		\begin{equation} \label{eq: rel per to rel per smooth}
				R \Hom_{(\Lambda_{n})_{Y'_{\RP}}}(F_{Y'}, G_{Y'})
			\to
				R \Hom_{(\Lambda_{n})_{Y_{\RPS} / Y'}}(\alpha_{\ast} F_{Y'}, \alpha_{\ast} G_{Y'})
		\end{equation}
	in the derived category of $\Lambda_{n}$-modules,
	where $F_{Y'}$ and $G_{Y'}$ are the restrictions to $Y'$.
	Let $M(F_{Y'})$ be Mac Lane's resolution of $F_{Y'}$ (\cite{Mac57}).
	Its homogeneous part at any degree is
	a direct summand of a direct sum of sheaves of the form $\Lambda_{n}[F_{Y'}^{m}]$
	for various $m \ge 0$,
	where $\Lambda_{n}[F_{Y'}^{m}]$ is the sheafification of the presheaf
	that assigns to each relatively perfect $Y'$-scheme $Y''$
	the free $\Lambda_{n}$-module generated by the set $F_{Y'}^{m}(Y'') = F^{m}(Y'')$.
	By assumption on $F$,
	we know that both $F_{Y'}^{m}$ and $\alpha_{\ast} F_{Y'}^{m}$ are representable by
	the relatively perfectly smooth $Y'$-scheme $F^{m} \times_{Y} Y'$.
	Hence the both sides of \eqref{eq: rel per to rel per smooth} can be written
	in terms of cohomology complexes of $F_{Y'}^{m}$ for various $m$.
	The morphism
		\[
				R \Gamma((F_{Y'}^{m})_{\RP}, G_{Y'})
			\to
				R \Gamma((F_{Y'}^{m})_{\RPS}, \alpha_{\ast} G_{Y'})
		\]
	is an isomorphism by the Leray spectral sequence.
	This implies the result.
\end{proof}

Let $D_{0}(Y_{\RPS}, \Lambda_{n})$ be the image of
$D_{0}(Y_{\RP}, \Lambda_{n})$ under $\alpha_{\ast}$.
We denote the image of each object $F \in D_{0}(Y_{\RPS}, \Lambda_{n})$
by the same letter $F$.

\begin{Cor}
	Assume that $Y$ is smooth and purely of dimension $d$
	and set $r = r_{0} + d$.
	The functor $\alpha_{\ast}$ gives an equivalence of categories
	$D_{0}(Y_{\RP}, \Lambda_{n}) \isomto D_{0}(Y_{\RPS}, \Lambda_{n})$.
	The derived sheaf-Hom functor
		\[
				D_{Y_{\RPS}}
			=
				R \sheafhom_{(\Lambda_{n})_{Y_{\RPS}}}(\;\cdot\;, \nu_{n}(r))
		\]
	for $D(Y_{\RPS}, \Lambda_{n})$ gives an auto-equivalence
	on $D_{0}(Y_{\RPS}, \Lambda_{n})$ with inverse itself.
\end{Cor}


\section{Formulation of the duality}
Let $k$ and $r_{0}$ be as above.
From now on, for a smooth $k$-scheme $Y$,
we will use $Y_{\RPS}$ and not $Y_{\RP}$,
so we write $D(Y, \Lambda_{n})$,
$D_{Y} = R \sheafhom_{(\Lambda_{n})_{Y}}(\;\cdot\;, \nu_{n}(r))$
etc.\ omitting the subscripts $\RPS$.

Let $K$ be a henselian discrete valuation field of characteristic $0$
whose residue field is $k$.
We denote the ring of integers of $K$ by $\Order_{K}$
and its maximal ideal by $\ideal{p}_{K}$.

Let $A$ be a flat $\Order_{K}$-algebra of finite type
and $\Hat{A}$ its $A \ideal{p}_{K}$-adic completion.
Write $A_{K} = A \tensor_{\Order_{K}} K$.
Let $R = A \tensor_{\Order_{K}} k$.
For a flat relatively perfect $R$-algebra $R'$,
we denote its canonical lifting over $\Hat{A}$ by $R'_{\Hat{A}}$ \cite[Definition 1]{Kat82}.
It is characterized as a unique complete $\Hat{A}$-algebra flat over $\Order_{K}$ such that
$R'_{\Hat{A}} \tensor_{\Order_{K}} k$ is isomorphic to $R'$ over $R$.
For any $n \ge 0$, the $A / A \ideal{p}_{K}^{n}$-algebra
$R'_{\Hat{A}} / R'_{\Hat{A}} \ideal{p}_{K}^{n}$ is flat and formally \'etale (\cite[Lemma 1]{Kat82}).
For another flat relatively perfect $R$-algebra $R''$,
we have
	\[
			\Hom_{\Hat{A}}(R'_{\Hat{A}}, R''_{\Hat{A}})
		\isomto
			\Hom_{R}(R', R'').
	\]
Write $R'_{\Hat{A}_{K}} = R'_{\Hat{A}} \tensor_{\Order_{K}} K$.
We have a commutative diagram with cocartesian squares
	\[
		\begin{CD}
				K
			@>>>
				A_{K}
			@>>>
				\Hat{A}_{K}
			@>>>
				R'_{\Hat{A}_{K}}
			\\
			@AAA
			@AAA
			@AAA
			@AAA
			\\
				\Order_{K}
			@>>>
				A
			@>>>
				\Hat{A}
			@>>>
				R'_{\Hat{A}}
			\\
			@VVV
			@VVV
			@VVV
			@VVV
			\\
				k
			@>>>
				R
			@=
				R
			@>>>
				R'.
		\end{CD}
	\]

\begin{Cor} \label{cor: relatively perfect schemes over X and Y}
	Under the above setting,
	let $A'$ be an $A$-algebra flat over $\Order_{K}$
	such that $A' \tensor_{\Order_{K}} k$ is flat relatively perfect over $R$.
	Then the $A' \ideal{p}_{K}$-adic completion $\Hat{A'}$ of $A'$ gives
	the canonical lifting of $A' \tensor_{\Order_{K}} k$ over $\Hat{A}$.
	The maps
		\[
				\Hom_{A}(A', R'_{\Hat{A}})
			\to
				\Hom_{\Hat{A}}(\Hat{A}', R'_{\Hat{A}})
			\to
				\Hom_{R}(A' \tensor_{\Order_{K}} k, R')
		\]
	are both bijective.
	In particular, a right adjoint of $A' \mapsto A' \tensor_{\Order_{K}} k$
	is given by $R' \mapsto R'_{\Hat{A}}$.
\end{Cor}

\begin{proof}
	This follows from the above characterization of canonical liftings.
\end{proof}

Let $X$ be a smooth $\Order_{K}$-scheme.
Let $U$ and $Y$ be its generic and special fibers, respectively.
Denote the natural inclusion morphisms by
	\[
			U
		\stackrel{j}{\into}
			X
		\stackrel{i}{\hookleftarrow}
			Y.
	\]
If $X = \Spec A$ is affine and $Y = \Spec R$,
then for an affine relatively perfectly smooth $Y$-scheme $Y' = \Spec R'$,
we denote $Y'_{\Hat{X}} = \Spec R'_{\Hat{A}}$
and
	$
			Y'_{\Hat{U}}
		=
			Y'_{\Hat{X}} \times_{\Order_{K}} K
		=
			\Spec R'_{\Hat{A}_{K}}
	$.
For a general smooth $X$,
let $X_{\RPS}$ be the category of $X$-schemes $X'$ flat over $\Order_{K}$
whose special fibers $X' \times_{X} Y$ are relatively perfectly smooth over $Y$.
Morphisms are $X$-scheme morphisms.
Endow $X_{\RPS}$ with the \'etale topology.
Let $\Tilde{X}_{\RPS}$ be the category of sheaves of sets on $X_{\RPS}$
and $\Tilde{Y}_{\RPS}$ similarly.
The reduction functor $X_{\RPS} \to Y_{\RPS}$, $X' \mapsto X' \times_{X} Y$,
defines a morphism of topologies.
It induces a pushforward functor $i_{\ast} \colon \Tilde{Y}_{\RPS} \to \Tilde{X}_{\RPS}$,
which has a left adjoint $i^{\ast} \colon \Tilde{X}_{\RPS} \to \Tilde{Y}_{\RPS}$.

\begin{Prop}
	The functor $i^{\ast}$ is exact.
	If $X$ is affine, then $i^{\ast} F$ for a sheaf $F$ on $X_{\RPS}$
	is given by the sheafification of the presheaf
	$Y' \mapsto \Gamma(Y'_{\Hat{X}}, F)$,
	where $Y'$ runs over affine relatively perfectly smooth $Y$-schemes.
\end{Prop}

\begin{proof}
	The second statement follows from
	Corollary \ref{cor: relatively perfect schemes over X and Y}.
	This description of $i^{\ast}$ shows that $i^{\ast}$ is exact when $X$ is affine.
	The general case follows.
\end{proof}

The above proposition shows that
the morphism $i \colon Y \into X$ induces a morphism of sites $Y_{\RPS} \to X_{\RPS}$.
Let $U_{\Et}$ be the category of $U$-schemes
endowed with the \'etale topology.
Denote $D(X, \Lambda_{n}) = D(X_{\RPS}, \Lambda_{n})$ and
$D(U, \Lambda_{n}) = D(U_{\Et}, \Lambda_{n})$.
Then we have morphisms of sites
	\begin{equation} \label{eq: diagram of generic and special fibers}
			U_{\Et}
		\stackrel{j}{\to}
			X_{\RPS}
		\stackrel{i}{\gets}
			Y_{\RPS}.
	\end{equation}
We consider the functor
	\[
			R \Psi
		=
			i^{\ast} R j_{\ast}
		\colon
			D(U, \Lambda_{n})
		\to
			D(Y, \Lambda_{n}).
	\]
We denote $R^{m} \Psi = i^{\ast} R^{m} j_{\ast}$ for $m \ge 0$.

\begin{Cor} \label{cor: nearby cycles as cohomology of canonical lifts}
	The functor $R \Psi$ is the right derived functor of $R^{0} \Psi$.
	If $X$ is affine, then $R^{m} \Psi F$
	for $F \in M(U_{\Et}, \Lambda)$ and $m \ge 0$
	is given by the sheafification of the presheaf
	$Y' \mapsto H^{m}(Y'_{\Hat{U}}, F)$,
	where $Y'$ runs over affine relatively perfectly smooth $Y$-schemes.
\end{Cor}

\begin{proof}
	This follows from the previous proposition.
\end{proof}

We have canonical morphisms
$R j_{\ast} F \tensor^{L} R j_{\ast} G \to R j_{\ast} (F \tensor^{L} G)$
in $D(X, \Lambda_{n})$ and hence
	\[
			R \Psi F \tensor^{L} R \Psi G
		\to
			R \Psi(F \tensor^{L} G)
	\]
in $D(Y, \Lambda_{n})$ functorial in $F, G \in D(U, \Lambda_{n})$.
Hence if we have a morphism $F \tensor^{L} G \to H$ in $D(U, \Lambda_{n})$,
then we have a canonical morphism
	\[
			R \Psi F \tensor^{L} R \Psi G
		\to
			R \Psi H.
	\]
For any integer $s$, we denote the $s$-th Tate twist of $\Lambda_{n}$ over $U$ by $\Lambda_{n} (s)$.
The following is the main theorem of this paper.

\begin{Thm} \label{thm: main}
	Let $X$ be a smooth $\Order_{K}$-scheme of relative dimension $d$.
	Let $U$ and $Y$ be its generic and special fibers, respectively.
	Set $r = r_{0} + d$.
	Let $s, t$ be integers with $s + t = r + 1$.
	\begin{enumerate}
		\item \label{item: main thm, trace morphism}
			There exists a canonical trace morphism
				\[
						\Tr
					\colon
						R^{r + 1} \Psi \Lambda_{n}(r + 1)
					\to
						\nu_{n}(r)
				\]
			of sheaves on $Y_{\RPS}$.
		\item \label{item: main thm, cohom dimension}
			The object $R \Psi \Lambda_{n}(s)$ is in $D_{0}(Y, \Lambda_{n})$
			and concentrated in degrees $[0, r + 1]$.
		\item \label{item: main thm, duality}
			The composite morphism
				\[
						R \Psi \Lambda_{n}(s)
					\tensor^{L}
						R \Psi \Lambda_{n}(t)
					\to
						R \Psi \Lambda_{n}(r + 1)
					\stackrel{\Tr}{\to}
						\nu_{n}(r)[- r - 1]
				\]
			induces a perfect duality between
			$R \Psi \Lambda_{n}(s)$ and $R \Psi \Lambda_{n}(t)$
			in $D_{0}(Y, \Lambda_{n})$
			via the dualizing functor $D_{Y}[- r - 1]$.
	\end{enumerate}
\end{Thm}

Theorem \ref{thm: mainthm in Introduction} is a consequence of this theorem.
We prove this theorem in the rest of the paper.


\section{Ind-smooth approximations of canonical liftings}

We continue working with the situation \eqref{eq: diagram of generic and special fibers}.
By Corollary \ref{cor: nearby cycles as cohomology of canonical lifts},
to describe the functor $R \Psi$,
we need to know the \'etale cohomology of the canonical liftings $Y'_{\Hat{U}}$.
For this, we use the following approximation method.

\begin{Prop} \label{prop: ind smooth approximation}
	Assume that $X = \Spec A$ is affine and $Y = \Spec R$ has a $p$-base.
	Let $\ideal{q} \subset R$ be a prime ideal and $\ideal{p} \subset A$ its inverse image.
	Let $R'$ be the local (resp.\ henselian local, resp.\ strict henselian local) ring
	of a relatively perfectly smooth $R$-algebra at some prime ideal containing $\ideal{q}$.
	
	Then there exists a local (resp.\ henselian local, resp.\ strict henselian local) $A_{\ideal{p}}$-algebra $A'$
	that can be written as a filtered direct limit of smooth $A$-algebras
	such that $A' \tensor_{\Order_{K}} k \cong R'$ as $R$-algebras
	and the pair $(A', A' \mathfrak{p}_{K})$ is henselian.
	The $A' \ideal{p}_{K}$-adic completion of $A'$
	is isomorphic to $R_{\Hat{A}}'$ as an $A$-algebra.
\end{Prop}

\begin{proof}
	All the relative perfections below are taken over $R$.
	By assumption, there exists a smooth $R$-algebra $R_{1}$
	and a prime ideal $\ideal{q}_{1} \subset R_{1}^{\RP}$
	such that $R'$ is the local (resp.\ henselian local, resp.\ strict henselian local) ring
	of $R_{1}^{\RP}$ at $\ideal{q}_{1}$.
	Taking $\Spec R_{1}$ smaller if necessary,
	we may assume that $R_{1}$ is \'etale over a polynomial ring $R_{2} = R[x_{1}, \dots, x_{m}]$.
	Then $R_{1}^{\RP}$ is \'etale over $R_{2}^{\RP}$.
	
	We show that there is a filtered direct limit of smooth $A$-algebras
	whose reduction $(\;\cdot\;) \tensor_{A} R$ is $R_{2}^{\RP}$.
	The relative perfection $\Spec R_{2}^{\RP}$ is given by the inverse limit of $G^{n}(\Spec R_{2})$ for $n \ge 0$,
	where $G$ is the Weil restriction functor for the absolute Frobenius morphism $\Spec R \to \Spec R$
	(\cite[1.6-1.8]{Kat86}).
	In particular, we only need to treat the case $m = 1$, so $R_{2} = R[x]$.
	Let $t_{1}, \dots, t_{r}$ be a $p$-base of $R$.
	Then $G(\Spec R_{2})$ is the affine $p^{r}$-space over $R$
	with coordinates $x_{i(1) \cdots i(r)}$, $0 \le i(1), \dots, i(r) \le p - 1$,
	and the $R$-morphism $G(\Spec R_{2}) = \Affine^{p^{r}}_{R} \to \Spec R_{2} = \Affine^{1}_{R}$ maps
	$(x_{i(1) \cdots i(r)})$ to $\sum x_{i(1) \cdots i(r)}^{p} t_{1}^{i(1)} \cdots t_{r}^{i(r)}$.
	In terms of rings, this is the $R$-algebra homomorphism
	$R[x] \to R[x_{i(1) \cdots i(r)} \,|\, 0 \le i(1), \dots, i(r) \le p - 1]$
	sending $x$ to $\sum x_{i(1) \cdots i(r)}^{p} t_{1}^{i(1)} \cdots t_{r}^{i(r)}$.
	We take a lifting of this morphism to $A$ by
	$\Affine^{p^{r}}_{A} \to \Affine^{1}_{A}$ mapping
	$(x_{i(1) \cdots i(r)})$ to $\sum x_{i(1) \cdots i(r)}^{p} t_{1}^{i(1)} \cdots t_{r}^{i(r)}$.
	Iterating, we can take a lifting of $G^{n + 1}(\Spec R_{2}) \to G^{n}(\Spec R_{2})$ to $A$
	by $\Affine^{p^{(n + 1) r}}_{A} \to \Affine^{p^{n r}}_{A}$ defined similarly.
	The inverse limit of these liftings for $n \ge 0$ gives
	a desired lifting of $\Spec R_{2}^{\RP}$.
	
	Let $A_{2}$ be such a lifting of $R_{2}^{\RP}$.
	Since $R_{1}^{\RP}$ is \'etale over $R_{2}^{\RP}$,
	we can take an \'etale $A_{2}$-algebra $A_{1}$ whose reduction is $R_{1}^{\RP}$.
	Let $\ideal{p}_{1} \subset A_{1}$ be the inverse image of $\ideal{q}_{1} \subset R_{1}^{\RP}$.
	Consider the local (resp.\ henselian local, resp.\ strict henselian local) ring $A_{1}'$ of $A_{1}$
	at $\ideal{p}_{1}$.
	The henselization of the pair $(A_{1}', A_{1}' \ideal{p}_{K})$ gives
	a desired local $A_{\ideal{p}}$-algebra $A'$.
	We have $\Hat{A}' \cong R_{\Hat{A}}'$
	by Corollary \ref{cor: relatively perfect schemes over X and Y}.
\end{proof}

We call the $A$-algebra $A'$ appearing in this proposition
an \emph{ind-smooth lifting} of $R'$ over $A$.
It is neither unique nor noetherian.

\begin{Prop} \label{prop: ind smooth approximation of cohomology}
	In the situation of Proposition \ref{prop: ind smooth approximation},
	for any torsion \'etale sheaf $F$ on $U_{\et}$ (pulled back to $U_{\Et}$),
	we have
		\[
				R \Gamma(A'_{K}, F)
			\isomto
				R \Gamma(R'_{\Hat{A}_{K}}, F).
		\]
	In the case of strictly henselian  $R'$,
	the isomorphic groups
	$H^{q}(A'_{K}, F) \cong H^{q}(R'_{\Hat{A}_{K}}, F)$
	for any $q$ give the stalk of $R^{q} \Psi F$ at the residue field of $R'$
	(if $F$ is a sheaf of $\Lambda_{n}$-modules).
\end{Prop}

\begin{proof}
	The first assertion follows from Fujiwara-Gabber's formal base change theorem
	(\cite[Expos\'e XX, \S 4.4]{ILO14}, \cite[Corollary 1.18 (2)]{BM18}).
	The second follows from
	Corollary \ref{cor: nearby cycles as cohomology of canonical lifts}.
\end{proof}

Since $A'$ is ind-smooth over $A$ and hence over $\Order_{K}$,
the study of $R \Gamma(A'_{K}, F)$ basically reduces to
Bloch-Kato's study of $p$-primary nearby cycles \cite{BK86}.


\section{Symbol maps and trace morphisms}
Let the notation be as in Theorem \ref{thm: main}.
We fix a prime element $\pi$ of $K$.
In this section, we will prove Theorem \ref{thm: main} \eqref{item: main thm, trace morphism}.
Slightly more generally, we will construct a certain morphism
	\begin{equation} \label{eq: boundary map}
			R^{q} \Psi \Lambda_{n}(q)
		\to
			\nu_{n}(q) \oplus \nu_{n}(q - 1)
	\end{equation}
of sheaves on $Y_{\RPS}$ such that
its composite with the projection onto the factor $\nu_{n}(q - 1)$ does not depend on $\pi$.

We need symbol maps adapted to our setting.
The connecting morphism for the Kummer exact sequence
$0 \to \Lambda_{n}(1) \to \Gm \to \Gm \to 0$ of sheaves on $U_{\Et}$
gives a morphism $i^{\ast} j_{\ast} \Gm \to R^{1} \Psi \Lambda_{n}(1)$ of sheaves on $Y_{\RPS}$.
By cup product, we define a morphism
	\begin{equation} \label{eq: symbol map}
			(i^{\ast} j_{\ast} \Gm)^{\tensor q}
		\to
			R^{q} \Psi \Lambda_{n}(q),
		\quad
			x_{1} \tensor \dots \tensor x_{q}
		\mapsto
			\{x_{1}, \dots, x_{q}\},
	\end{equation}
where $\tensor q$ means the $q$-th tensor power and
the $x_{i}$ are local sections of $i^{\ast} j_{\ast} \Gm$
(i.e.\ invertible elements of $R'_{\Hat{A}_{K}}$
for some relatively perfectly smooth $R$-algebra $R'$,
where $\Spec A$ is an affine open of $X$
and $R = A \tensor_{\Order_{K}} k$),
which we call the \emph{symbol map}.
By composing it with the inclusion $\Gm \into j_{\ast} \Gm$,
we have a morphism from $(i^{\ast} \Gm)^{\tensor q}$ to $R^{q} \Psi \Lambda_{n}(q)$.
The construction of the morphism \eqref{eq: boundary map} is given by the following,
which proves Theorem \ref{thm: main} \eqref{item: main thm, trace morphism}.

\begin{Prop} \label{prop: boundary map}
	The morphism
		\begin{gather*}
					(i^{\ast} \Gm)^{\tensor q} \oplus (i^{\ast} \Gm)^{\tensor q - 1}
				\to
					R^{q} \Psi \Lambda_{n}(q),
			\\
					(x_{1} \tensor \dots \tensor x_{q}, y_{1} \tensor \dots \tensor y_{q - 1})
				\mapsto
					\{x_{1}, \dots, x_{q}\} + \{y_{1}, \dots, y_{q - 1}, \pi\}
		\end{gather*}
	is surjective.
	(Note that the last component of the second symbol is $\pi \in \Gamma(Y, i^{\ast} j_{\ast} \Gm)$,
	which is not in $\Gamma(Y, i^{\ast} \Gm)$.)
	The composite of the reduction map and the dlog map
		\begin{gather*}
					(i^{\ast} \Gm)^{\tensor q} \oplus (i^{\ast} \Gm)^{\tensor q - 1}
				\to
					\Gm^{\tensor q} \oplus \Gm^{\tensor q - 1}
				\to
					\nu_{n}(q) \oplus \nu_{n}(q - 1)
		\end{gather*}
	factors through the quotient $R^{q} \Psi \Lambda_{n}(q)$.
	The obtained morphism
		\[
				R^{q} \Psi \Lambda_{n}(q)
			\to
				\nu_{n}(q) \oplus \nu_{n}(q - 1)
		\]
	followed by the projection onto the factor $\nu_{n}(q - 1)$ does not depend on $\pi$.
\end{Prop}

In this proposition, when $X = \Spec \Order_{K}$, $k$ is separably closed and $q = 1$,
the global section of the above morphism
	$
			i^{\ast} \Gm \oplus \Z
		\to
			R^{1} \Psi \Lambda_{n}(1)
	$
is $\Hat{\Order}_{K}^{\times} \oplus \Z \to \Hat{K}^{\times} / (\Hat{K}^{\times})^{p^{n}}$
given by $(x, n) \mapsto x \pi^{n}$.
This is indeed surjective.

\begin{proof}
	It is enough to check the statements for stalks.
	Hence we may assume that $X = \Spec A$ is affine
	and $Y = \Spec R$ has a $p$-base.
	Let $R'$ be the strict henselian local ring of a relatively perfectly smooth $R$-algebra
	at some prime ideal.
	Let $A'$ be an ind-smooth lifting of $R'$ over $A$ as in the previous section.
	By Proposition \ref{prop: ind smooth approximation of cohomology},
	we are reduced to proving the following:
	the homomorphism
		\begin{gather*}
					(A'^{\times})^{\tensor q} \oplus (A'^{\times})^{\tensor q - 1}
				\to
					H^{q}(A'_{K}, \Lambda_{n}(q))
			\\
					(x_{1} \tensor \dots \tensor x_{q}, y_{1} \tensor \dots \tensor y_{q - 1})
				\mapsto
					\{x_{1}, \dots, x_{q}\} + \{y_{1}, \dots, y_{q - 1}, \pi\}
		\end{gather*}
	is surjective;
	the composite of the reduction map and the dlog map
		\begin{gather*}
					(A'^{\times})^{\tensor q} \oplus (A'^{\times})^{\tensor q - 1}
				\to
					(R'^{\times})^{\tensor q} \oplus (R'^{\times})^{\tensor q - 1}
				\to
					\Gamma(R', \nu_{n}(q)) \oplus \Gamma(R', \nu_{n}(q - 1))
		\end{gather*}
	factors through the quotient $H^{q}(A'_{K}, \Lambda_{n}(q))$;
	and the obtained homomorphism
		\[
				H^{q}(A'_{K}, \Lambda_{n}(q))
			\to
				\Gamma(R', \nu_{n}(q)) \oplus \Gamma(R', \nu_{n}(q - 1))
		\]
	followed by the projection onto the factor $\Gamma(R', \nu_{n}(q - 1))$ does not depend on $\pi$.
	Since $A'$ is ind-smooth over $A$ and hence over $\Order_{K}$,
	these claims are reduced to \cite[Theorem (1.4) (i)]{BK86}.
\end{proof}

For $m \le n$, the endomorphisms of $\Lambda_{n}$ and $\nu_{n}(q)$ given by multiplication by $p^{n - m}$
factor as $\formal{p^{n - m}} \colon \Lambda_{m} \into \Lambda_{n}$
and $\formal{p^{n - m}} \colon \nu_{m}(q) \into \nu_{n}(q)$,
so that we have an exact sequence
$0 \to \nu_{m}(q) \to \nu_{n}(q) \to \nu_{n - m}(q) \to 0$ (\cite[(4.1.8)]{Kat86}).
Later we will use the following.

\begin{Prop} \label{prop: trace morphism and formal multiplication by p}
	We have a commutative diagram
		\[
			\begin{CD}
					R^{q} \Psi \Lambda_{m}(q)
				@>>>
					\nu_{m}(q) \oplus \nu_{m}(q - 1)
				\\
				@VV \formal{p^{n - m}} V
				@V \formal{p^{n - m}} VV
				\\
					R^{q} \Psi \Lambda_{n}(q)
				@>>>
					\nu_{n}(q) \oplus \nu_{n}(q - 1),
			\end{CD}
		\]
	where the horizontal morphisms are given by \eqref{eq: boundary map}
	for $m$ and $n$.
\end{Prop}

\begin{proof}
	Consider the following diagram (commutativity to be discussed soon):
		\[
			\begin{CD}
					(i^{\ast} \Gm)^{\tensor q} \oplus (i^{\ast} \Gm)^{\tensor q - 1}
				@>>>
					R^{q} \Psi \Lambda_{n}(q)
				@>>>
					\nu_{n}(q) \oplus \nu_{n}(q - 1)
				\\
				@|
				@VV \text{can} V
				@V \text{can} VV
				\\
					(i^{\ast} \Gm)^{\tensor q} \oplus (i^{\ast} \Gm)^{\tensor q - 1}
				@>>>
					R^{q} \Psi \Lambda_{m}(q)
				@>>>
					\nu_{m}(q) \oplus \nu_{m}(q - 1).
				\\
				@VV p^{n - m} V
				@VV \formal{p^{n - m}} V
				@V \formal{p^{n - m}} VV
				\\
					(i^{\ast} \Gm)^{\tensor q} \oplus (i^{\ast} \Gm)^{\tensor q - 1}
				@>>>
					R^{q} \Psi \Lambda_{n}(q)
				@>>>
					\nu_{n}(q) \oplus \nu_{n}(q - 1).
			\end{CD}
		\]
	The left three horizontal morphisms (the symbol maps) are all surjective
	by Proposition \ref{prop: boundary map}.
	The left two squares are commutative by the construction of the symbol map.
	The total (or outer) square omitting the central term $R^{q} \Psi \Lambda_{m}(q)$ is commutative
	since $\formal{p^{n - m}} \compose \text{can} = p^{n - m}$.
	From these, the commutativity of the right lower square follows
	by a diagram chase.
\end{proof}


\section{Mod $p$ case I: filtrations and duality for $\gr^{0}$}
Let the notation be as in Theorem \ref{thm: main}.
Within this and the next sections, we will prove Theorem \ref{thm: main} for the case $n = 1$.
We fix a prime element $\pi$ of $K$.
Let $q \ge 0$.
Recall our notation $\Lambda = \Z / p \Z$.
As in \cite[(1.2)]{BK86}, we define a filtration on the sheaf $R^{q} \Psi \Lambda(q)$
using the symbol map \eqref{eq: symbol map} as follows.
For $m \ge 1$, define $U^{m} R^{q} \Psi \Lambda(q)$ to be the subsheaf of $R^{q} \Psi \Lambda(q)$
generated by local sections of the form $\{x_{1}, \dots, x_{q}\}$
such that $x_{1} - 1 \in \pi^{m} i^{\ast} \Ga$.
Let
	\[
			\gr^{m} R^{q} \Psi \Lambda(q)
		=
			\begin{cases}
					R^{q} \Psi \Lambda(q) / U^{1} R^{q} \Psi \Lambda(q)
				&
					\text{if } m = 0,
				\\
					U^{m} R^{q} \Psi \Lambda(q) / U^{m + 1} R^{q} \Psi \Lambda(q)
				&
					\text{if } m \ge 1.
			\end{cases}
	\]
For $m \ge 1$, define a morphism $\rho_{m}$ from the direct sum of
$i^{\ast} \Ga \tensor i^{\ast} \Gm^{\tensor q - 1}$
and $i^{\ast} \Ga \tensor i^{\ast} \Gm^{\tensor q - 2}$
to $U^{m} R^{q} \Psi \Lambda(q)$ by
	\[
			x \tensor y_{1} \tensor \dots \tensor y_{q - 1}
		\mapsto
			\{1 + x \pi^{m}, y_{1}, \dots, y_{q - 1}\}
	\]
and
	\[
			x \tensor y_{1} \tensor \dots \tensor y_{q - 2}
		\mapsto
			\{1 + x \pi^{m}, y_{1}, \dots, y_{q - 2}, \pi\}.
	\]
The reduction map and the dlog map define surjections
	\[
			i^{\ast} \Ga \tensor i^{\ast} \Gm^{\tensor q - 1}
		\onto
			\Ga \tensor \Gm^{\tensor q - 1}
		\onto
			\Omega_{Y}^{q - 1}
	\]
and similar surjections with $q - 1$ replaced by $q - 2$.
Let $e$ be the absolute ramification index of $K$
and set $e' = p e / (p - 1)$.

\begin{Prop} \label{prop: calculation of graded pieces} \mbox{}
	\begin{enumerate} \setcounter{enumi}{-1}
		\item \label{item: morphism from differential to Gr}
			For $m \ge 1$, the morphism $\rho_{m}$ factor through
				\[
						\Omega_{Y}^{q - 1} \oplus \Omega_{Y}^{q - 2}
					\to
						\gr^{m} R^{q} \Psi \Lambda(q).
				\]
		\item
			The morphism \eqref{eq: boundary map} for $n = 1$ is surjective
			and induces an isomorphism
				\[
						\gr^{0} R^{q} \Psi \Lambda(q)
					\cong
						\nu(q) \oplus \nu(q - 1).
				\]
		\item
			If $1 \le m < e'$ and $p \nmid m$, then the morphism in \eqref{item: morphism from differential to Gr}
			induces an isomorphism
				\[
						\gr^{m} R^{q} \Psi \Lambda(q)
					\cong
						\Omega_{Y}^{q - 1}.
				\]
		\item
			If $1 \le m < e'$ and $p \mid m$, then the morphism in \eqref{item: morphism from differential to Gr}
			and the differential $d$ induce an isomorphism
				\[
						\gr^{m} R^{q} \Psi \Lambda(q)
					\cong
						d \Omega_{Y}^{q - 1} \oplus d \Omega_{Y}^{q - 2}.
				\]
		\item
			If $m \ge e'$, then
				\[
						U^{m} R^{q} \Psi \Lambda(q)
					=
						0.
				\]
	\end{enumerate}
\end{Prop}

\begin{proof}
	This reduces to \cite[Corollary (1.4.1)]{BK86}
	by the same method as the proof of Proposition \ref{prop: boundary map}.
\end{proof}

\begin{Prop} \label{prop: main thm, cohom dimension is true for n equal to one}
	The statement of Theorem \ref{thm: main} \eqref{item: main thm, cohom dimension} is true for $n = 1$.
\end{Prop}

\begin{proof}
	We may assume that $K$ contains a primitive $p$-th root of unity $\zeta_{p}$
	since $[K(\zeta_{p}) : K]$ is prime to $p$.
	Then the result follows from Proposition \ref{prop: calculation of graded pieces}.
\end{proof}

Thus we have the morphism
	\[
			R \Psi \Lambda(s)
		\tensor^{L}
			R \Psi \Lambda(t)
		\to
			R \Psi \Lambda(r + 1)
		\stackrel{\Tr}{\to}
			\nu(r)[- r - 1]
	\]
stated in Theorem \ref{thm: main} \eqref{item: main thm, duality}
in the case $n = 1$,
where $s, t$ are integers with $s + t = r + 1$.
We want to prove that it induces a perfect duality.
In the rest of this section, we work with $\Lambda = \Z / p \Z$-coefficients.
As above, we may assume that $\zeta_{p} \in K$.
With the choice of $\zeta_{p}$,
we may identify all the Tate twists $\Lambda(q)$ with $\Lambda$ in a compatible way.
Let $\mathcal{E} = R \Psi \Lambda$ ($\cong R \Psi \Lambda(q)$ for any $q$).
The above morphism  may be written as
	\begin{equation} \label{eq: duality morphism mod p with root of unity}
			\mathcal{E} \tensor^{L} \mathcal{E}
		\to
			\nu(r)[-r - 1],
	\end{equation}
which is independent of the integers $s, t$.
We have the above filtrations $U^{m} H^{q} \mathcal{E}$ and graded pieces $\gr^{m} H^{q} \mathcal{E}$ for any $q$.
For any $s$, define $\tau_{\ge s}' \mathcal{E}$ to be the canonical mapping cone of the natural morphism
$U^{1} H^{s} \mathcal{E}[-s] \to \tau_{\ge s} \mathcal{E}$
and $\tau_{\le s}' \mathcal{E}$ to be the canonical mapping fiber of the natural morphism
$\tau_{\le s} \mathcal{E} \to \gr^{0} H^{s} \mathcal{E}[-s]$.
We have distinguished triangles
	\begin{gather*}
				U^{1} H^{s} \mathcal{E}[-s]
			\to
				\tau_{\ge s} \mathcal{E}
			\to
				\tau_{\ge s}' \mathcal{E},
		\\
				\tau_{\le s}' \mathcal{E}
			\to
				\tau_{\le s} \mathcal{E}
			\to
				\gr^{0} H^{s} \mathcal{E}[-s],
		\\
				\gr^{0} H^{s} \mathcal{E}[-s]
			\to
				\tau_{\ge s}' \mathcal{E}
			\to
				\tau_{\ge s + 1} \mathcal{E},
		\\
				\tau_{\le s - 1} \mathcal{E}
			\to
				\tau_{\le s}' \mathcal{E}
			\to
				U^{1} H^{s} \mathcal{E}[-s],
	\end{gather*}
where the latter two are truncation triangles.

\begin{Prop} \label{prop: splitting pairing to unit part and gr zero part}
	Let us abbreviate $R \sheafhom_{Y}$ as $[\;\cdot\;, \;\cdot\;]$.
	There exists a unique set of morphisms
		\begin{gather*}
					\tau_{\ge s + 1} \mathcal{E} \tensor^{L} \tau_{\le t}' \mathcal{E}
				\to
					\nu(r)[- r -1],
			\\
					\tau_{\ge s}' \mathcal{E} \tensor^{L} \tau_{\le t} \mathcal{E}
				\to
					\nu(r)[- r -1],
			\\
					U^{1} H^{s} \mathcal{E} \tensor^{L} U^{1} H^{t + 1} \mathcal{E}
				\to
					\nu(r)[1],
			\\
					\gr^{0} H^{s} \mathcal{E} \tensor^{L} \gr^{0} H^{t} \mathcal{E}
				\to
					\nu(r)
		\end{gather*}
	for integers $s, t$ with $s + t = r + 1$
	that reduce to \eqref{eq: duality morphism mod p with root of unity} for $s < 0$
	and give morphisms of distinguished triangles from
		\begin{equation} \label{eq: modified truncation}
				U^{1} H^{s} \mathcal{E}[-s]
			\to
				\tau_{\ge s} \mathcal{E}
			\to
				\tau_{\ge s}' \mathcal{E}
		\end{equation}
	to the shift $[- r - 1]$ of 
		\begin{equation} \label{eq: dual of truncation for modified truncation}
				\bigl[
					U^{1} H^{t + 1} \mathcal{E}[- t - 1], \nu(r)
				\bigr]
			\to
				[\tau_{\le t + 1}' \mathcal{E}, \nu(r)]
			\to
				[\tau_{\le t} \mathcal{E}, \nu(r)]
		\end{equation}
	and from
		\begin{equation} \label{eq: truncation for modified truncation}
				\gr^{0} H^{s} \mathcal{E}[-s]
			\to
				\tau_{\ge s}' \mathcal{E}
			\to
				\tau_{\ge s + 1} \mathcal{E}
		\end{equation}
	to the shift $[-r - 1]$ of
		\begin{equation} \label{eq: dual of modified truncation}
				\bigl[
					\gr^{0} H^{t} \mathcal{E}[-t], \nu(r)
				\bigr]
			\to
				[\tau_{\le t} \mathcal{E}, \nu(r)]
			\to
				[\tau_{\le t}' \mathcal{E}, \nu(r)].
		\end{equation}
\end{Prop}

\begin{proof}
	First we show that
		\begin{equation} \label{eq: condition to extend the pairing, from units}
			\begin{aligned}
				&		\Hom \bigl(
							U^{1} H^{s} \mathcal{E}[-s],
							[\tau_{\le t} \mathcal{E}, \nu(r)][- r - 1]
						\bigr)
				\\
				&	=
						\Hom \bigl(
							U^{1} H^{s} \mathcal{E}[-s],
							[\tau_{\le t} \mathcal{E}, \nu(r)][- r - 2]
						\bigr)
					=
						0.
			\end{aligned}
		\end{equation}
	Since $[\tau_{\le t} \mathcal{E}, \nu(r)]$ is concentrated in degrees $\ge -t = s - r - 1$,
	the second term is zero simply by a degree reason.
	By the same reasoning, the first term is equal to
		\[
				\Hom \bigl(
					U^{1} H^{s} \mathcal{E},
					\sheafhom(H^{t} \mathcal{E}, \nu(r))
				\bigr)
			=
				\Hom \bigl(
					H^{t} \mathcal{E},
					\sheafhom(U^{1} H^{s} \mathcal{E}, \nu(r))
				\bigr).
		\]
	The sheaf $U^{1} H^{s} \mathcal{E}$ is a finite successive extension of locally free $\Order_{Y}$-modules of finite rank
	by Proposition \ref{prop: calculation of graded pieces}.
	Hence $\sheafhom(U^{1} H^{s} \mathcal{E}, \nu(r)) = 0$
	by \cite[Theorem 3.2 (ii)]{Kat86}.
	This proves \eqref{eq: condition to extend the pairing, from units}.
	
	Next we show that
		\begin{equation} \label{eq: condition to extend the pairing, from gr zero}
			\begin{aligned}
				&		\Hom \bigl(
							\gr^{0} H^{s} \mathcal{E}[-s],
							[\tau_{\le t}' \mathcal{E}, \nu(r)][- r - 1]
						\bigr)
				\\
				&	=
						\Hom \bigl(
							\gr^{0} H^{s} \mathcal{E}[-s],
							[\tau_{\le t}' \mathcal{E}, \nu(r)][- r - 2]
						\bigr)
					=
						0.
			\end{aligned}
		\end{equation}
	The same reasoning as above shows that the second term is zero and
	the first term is equal to
		\[
			\Hom \bigl(
				\gr^{0} H^{s} \mathcal{E},
				\sheafhom(U^{1} H^{t} \mathcal{E}, \nu(r))
			\bigr)
		\]
	since $H^{t}(\tau_{\le t}' \mathcal{E}) = U^{1} H^{t} \mathcal{E}$.
	We have $\sheafhom(U^{1} H^{t} \mathcal{E}, \nu(r)) = 0$ by \cite[Theorem 3.2 (ii)]{Kat86}
	since $U^{1} H^{t} \mathcal{E}$ is a finite successive extension of locally free $\Order_{Y}$-modules of finite rank
	by Proposition \ref{prop: calculation of graded pieces}.
	This proves \eqref{eq: condition to extend the pairing, from gr zero}.
	
	Now we prove the proposition by induction on $s$.
	There is nothing to do for $s < 0$.
	Fix integers $s_{0}, t_{0}$ with $s_{0} + t_{0} = r + 1$.
	Assume that there exists a unique set of morphisms
	as stated for $s, t$ with $s + t = r + 1$ and $s \le s_{0}$.
	We want to prove the same for $s = s_{0} + 1$ and $t = t_{0} - 1$.
	By assumption, we have a morphism
	from $\tau_{\ge s_{0} + 1} \mathcal{E}$ to the shift $[- r - 1]$ of $[\tau_{\le t_{0}}' \mathcal{E}, \nu(r)]$.
	This gives a morphism from the middle term of
	\eqref{eq: modified truncation} to the middle term of the shift $[-r - 1]$ of
	\eqref{eq: dual of truncation for modified truncation}
	for $s = s_{0} + 1$.
	By \eqref{eq: condition to extend the pairing, from units},
	this morphism uniquely extends to a morphism of distinguished triangles
	from \eqref{eq: modified truncation} to the shift $[-r - 1]$ of
	\eqref{eq: dual of truncation for modified truncation} for $s = s_{0} + 1$.
	In particular, we have a morphism
	from $\tau_{\ge s_{0} + 1}' \mathcal{E}$ to the shift $[- r - 1]$ of $[\tau_{\le t_{0} - 1} \mathcal{E}, \nu(r)]$.
	This gives a morphism from the middle term of
	\eqref{eq: truncation for modified truncation} to the middle term of the shift $[-r - 1]$ of
	\eqref{eq: dual of modified truncation}
	for $s = s_{0} + 1$.
	By \eqref{eq: condition to extend the pairing, from gr zero},
	this morphism uniquely extends to a morphism of distinguished triangles
	from \eqref{eq: truncation for modified truncation} to the shift $[-r - 1]$ of
	\eqref{eq: dual of modified truncation} for $s = s_{0} + 1$.
	This proves the induction step,
	and hence the proposition itself.
\end{proof}

\begin{Prop} \label{prop: duality for gr zero}
	The morphism
		$
				\gr^{0} H^{s} \mathcal{E} \tensor^{L} \gr^{0} H^{t} \mathcal{E}
			\to
				\nu(r)
		$
	in Proposition \ref{prop: splitting pairing to unit part and gr zero part}
	gives a perfect duality.
\end{Prop}

\begin{proof}
	The stated morphism factors through $\gr^{0} H^{s} \mathcal{E} \tensor \gr^{0} H^{t} \mathcal{E}$.
	We have $\gr^{0} H^{s} \mathcal{E} \cong \nu(s) \oplus \nu(s - 1)$
	and $\gr^{0} H^{t} \mathcal{E} \cong \nu(t) \oplus \nu(t - 1)$
	by Proposition \ref{prop: calculation of graded pieces}.
	Hence the stated morphism gives rise to a pairing
		\[
					\nu(s) \oplus \nu(s - 1)
				\times
					\nu(t) \oplus \nu(t - 1)
			\to
				\nu(r).
		\]
	By \cite[Theorem 3.2 (i)]{Kat86},
	it is enough to show that this pairing is given by
		\begin{equation} \label{eq: explicit pairing in gr zero}
				\bigl(
					(\omega, \omega'), (\tau, \tau')
				\bigr)
			\mapsto
				\pm \omega \wedge \tau' \pm \omega' \wedge \tau.
		\end{equation}
	We may assume that $X = \Spec A$ is affine and $Y = \Spec R$ has a $p$-base.
	The composite of the natural surjection
	$H^{s} \mathcal{E} \tensor H^{t} \mathcal{E} \onto \gr^{0} H^{s} \mathcal{E} \tensor \gr^{0} H^{t} \mathcal{E}$
	and the morphism $\gr^{0} H^{s} \mathcal{E} \tensor \gr^{0} H^{t} \mathcal{E} \to \nu(r)$
	is the morphism $H^{s} \mathcal{E} \tensor H^{t} \mathcal{E} \to \nu(r)$
	induced by \eqref{eq: duality morphism mod p with root of unity}.
	Let $R'$ be the strict henselian local ring
	of a relatively perfectly smooth $R$-algebra at a prime ideal.
	We want to describe our pairing on $R'$-points.
	The map on $R'$-points of the morphism $H^{s} \mathcal{E} \tensor H^{t} \mathcal{E} \to \nu(r)$ is of the form
		\[
					H^{s}(R_{\Hat{A}_{K}}', \Lambda(s))
				\tensor
					H^{t}(R_{\Hat{A}_{K}}', \Lambda(t))
			\to
				H^{r + 1}(R_{\Hat{A}_{K}}', \Lambda(r + 1))
			\to
				\Gamma(R', \nu(r)).
		\]
	Let $A'$ be an ind-smooth lifting of $R'$ over $A$.
	By Proposition \ref{prop: ind smooth approximation of cohomology},
	the above map can be written as
		\[
					H^{s}(A_{K}', \Lambda(s))
				\tensor
					H^{t}(A_{K}', \Lambda(t))
			\to
				H^{r + 1}(A_{K}', \Lambda(r + 1))
			\to
				\Gamma(R', \nu(r)).
		\]
	The groups in the first term are generated by symbols
	by \cite[Theorem (1.4)]{BK86}.
	The first map is given by concatenation of symbols
	and the second described by the paragraph after \cite[Corollary (1.4.1)]{BK86}.
	By an easy computation of symbols and dlog forms,
	we see that our pairing is indeed given the formula \eqref{eq: explicit pairing in gr zero}.
	This proves the proposition.
\end{proof}


\section{Mod $p$ case II: duality for $U^{1}$}
\label{sec: Mod p case: duality for U one}
We keep the notation from the last section.
In particular, we fix a prime element $\pi$ of $K$
and a primitive $p$-th root of unity $\zeta_{p} \in K$,
and we work with $\Lambda$-coefficients.
To treat the part $U^{1} H^{s} \mathcal{E} \tensor^{L} U^{1} H^{t + 1} \mathcal{E} \to \nu(r)[1]$
of Proposition \ref{prop: splitting pairing to unit part and gr zero part},
it is convenient to use the Zariski topology in addition to the \'etale topology.

Let $Y_{\RPSZ}$ be the category of relatively perfectly smooth $Y$-schemes
endowed with the Zariski topology.
Let $\varepsilon \colon Y_{\RPS} \to Y_{\RPSZ}$ be the morphism defined by the identity functor.
Recall from \cite[(3.1.4), (3.1.5)]{Kat86} that there are exact sequences
	\begin{gather*}
				0
			\to
				\nu(r)
			\to
				\Omega_{Y, d = 0}^{r}
			\overset{C - 1}{\longrightarrow}
				\Omega_{Y}^{r}
			\to
				0,
		\\
				0
			\to
				\nu(r)
			\to
				\Omega_{Y}^{r}
			\overset{C^{-1} - 1}{\longrightarrow}
				\Omega_{Y}^{r} / d \Omega_{Y}^{r - 1}
			\to
				0
	\end{gather*}
in $Y_{\RPS}$, where $C$ is the Cartier operator.
Since $r = r_{0} + d$ is the number of elements in local $p$-bases of $Y$,
we have $\Omega_{Y}^{r + 1} = 0$,
so $\Omega_{Y, d = 0}^{r} = \Omega_{Y}^{r}$
and $C$ is an endomorphism of $\Omega_{Y}^{r}$.
We can view $\nu(r)$ as a sheaf on $Y_{\RPSZ}$,
which is the kernel of the endomorphism $C - 1$ on $\Omega_{Y}^{r}$.
Define a sheaf $\xi(r)$ on $Y_{\RPSZ}$ to be
the cokernel of the endomorphism $C - 1$ on $\Omega_{Y}^{r}$ over $Y_{\RPSZ}$.
The exact sequence $0 \to \nu(r) \to \Omega_{Y}^{r} \overset{C - 1}{\to} \Omega_{Y}^{r} \to 0$ over $Y_{\RPS}$
shows that $R^{n} \varepsilon_{\ast} \nu(r) = 0$ for $n \ge 2$
and $R^{1} \varepsilon_{\ast} \nu(r) = \xi(r)$, and
defines a morphism
	\[
			R \varepsilon_{\ast} \nu(r)
		\to
			\xi(r)[-1]
	\]
in $D(Y_{\RPSZ}, \Lambda)$.
For any $M \in D_{0}(Y_{\RPS}, \Lambda)$, 
the isomorphism $\varepsilon^{\ast} R \varepsilon_{\ast} M \cong M$,
the sheafified derived adjunction and the above morphism define a morphism
	\begin{align*}
				R \varepsilon_{\ast}
				R \sheafhom_{Y_{\RPS}}(M, \nu(r))
		&	\cong
				R \sheafhom_{Y_{\RPSZ}}(R \varepsilon_{\ast} M, R \varepsilon_{\ast} \nu(r))
		\\
		&	\to
				R \sheafhom_{Y_{\RPSZ}}(R \varepsilon_{\ast} M, \xi(r))[-1]
	\end{align*}
in $D(Y_{\RPSZ}, \Lambda)$.

\begin{Prop} \label{prop: etale duality pushes to Zariski duality}
	The above morphism is an isomorphism.
\end{Prop}

\begin{proof}
	We need to show that
	$R \sheafhom_{Y_{\RPSZ}}(R \varepsilon_{\ast} M, \nu(r)) = 0$.
	We may assume that $Y$ is affine with a $p$-base and $M = \Ga$,
	and it is enough to show that $R \Hom_{Y_{\RPSZ} / Y'}(\Ga, \nu(r))$ is zero
	for any relatively perfectly smooth affine $Y$-scheme $Y'$.
	By \cite[Theorem (2.1)]{Blo86} and \cite[Theorem 8.3]{GL00},
	we know that Zariski cohomology with coefficients in $\nu(r)$ is homotopy invariant.
	Hence the natural morphism from $R \Gamma(Y'_{\zar}, \nu(r))$ to
	$R \Gamma((\Affine_{Y'}^{m})_{\zar}, \nu(r))$
	is invertible for any $m \ge 1$.
	This implies that the natural morphism from $R \Gamma(Y'_{\zar}, \nu(r))$ to
	$R \Gamma((\Affine_{Y'}^{m})^{\RP}_{\zar}, \nu(r))$
	(the Zariski cohomology of the relative perfection of $\Affine_{Y'}^{m}$)
	is invertible for any $m \ge 1$
	since $(\Affine_{Y'}^{m})^{\RP}$ is an inverse limit of affine spaces over $Y'$.
	Hence, by using Mac Lane's resolution of $\Ga$ as in the proof of Proposition \ref{prop: from RP to RPS},
	we know that the morphism from
	$R \Hom_{Y_{\RPSZ} / Y'}(0, \nu(r))$ (which is zero)
	to $R \Hom_{Y_{\RPSZ} / Y'}(\Ga, \nu(r))$ induced by $\Ga \to 0$ is invertible.
	The result then follows.
\end{proof}

\begin{Prop}
	Let $M \in M(Y_{\RPS}, \Lambda)$ be a sheaf admitting a finite filtration
	whose graded pieces are isomorphic to relative perfections of coherent sheaves on $Y$
	locally free of finite rank.
	Note that $R \varepsilon_{\ast} M = \varepsilon_{\ast} M$.
	View $M$ also as a sheaf on $Y_{\RPSZ}$.
	Then
		\begin{gather*}
				R \sheafhom_{Y_{\RPS}}(M, \nu(r)),
			\quad
				R \varepsilon_{\ast} R \sheafhom_{Y_{\RPS}}(M, \nu(r)),
			\\
				R \sheafhom_{Y_{\RPSZ}}(M, \xi(r))[-1]
		\end{gather*}
	are all concentrated in degree $1$.
\end{Prop}

\begin{proof}
	This follows from \cite[Theorem 3.2 (ii)]{Kat86}
	and Proposition \ref{prop: etale duality pushes to Zariski duality}.
\end{proof}

Let $\mathcal{E}' = R \varepsilon_{\ast} \mathcal{E}$;
cf.\ the paragraph before \cite[Theorem (6.7)]{BK86}.

\begin{Prop} \label{prop: Zariski stalks of nearby cycles}
	Assume that $X = \Spec A$ is affine and $Y = \Spec R$ has a $p$-base.
	Let $R'$ be the local ring of a relatively perfectly smooth $R$-algebra at a prime ideal.
	Then for any $q$, the stalk of the Zariski sheaf $H^{q} \mathcal{E}'$ at the closed point of $\Spec R'$
	is given by $H^{q}(R'_{\Hat{A}_{K}}, \Lambda(q))$ (cohomology in the \'etale topology).
\end{Prop}

\begin{proof}
	The stalk is given by $H^{q}(R'_{\et}, i^{\ast} R j_{\ast} \Lambda(q))$.
	Since the pair $(R'_{\Hat{A}}, R'_{\Hat{A}} \ideal{p}_{K})$ is henselian,
	this group is isomorphic to $H^{q}(R'_{\Hat{A}}, R j_{\ast} \Lambda(q))$
	by Gabber's affine analog of proper base change \cite[Theorem 1]{Gab94}.
	This final group is isomorphic to $H^{q}(R'_{\Hat{A}_{K}}, \Lambda(q))$.
\end{proof}

Define a filtration on $H^{q} \mathcal{E}'$ in the same way as in the case of $H^{q} \mathcal{E}$.
As in Proposition \ref{prop: calculation of graded pieces},
we have the following.

\begin{Prop} \label{prop: calculation of graded pieces in Zariski} \mbox{}
	\begin{enumerate}
		\item
			We have
				\[
						\gr^{0} H^{q} \mathcal{E}'
					\cong
						\nu(q) \oplus \nu(q - 1).
				\]
		\item
			If $1 \le m < e'$ and $p \nmid m$, then
				\[
						\gr^{m} H^{q} \mathcal{E}'
					\cong
						\Omega_{Y}^{q - 1}.
				\]
		\item
			If $1 \le m < e'$ and $p \mid m$, then
				\[
						\gr^{m} H^{q} \mathcal{E}'
					\cong
						d \Omega_{Y}^{q - 1} \oplus d \Omega_{Y}^{q - 2}.
				\]
		\item
			We have
				\[
						U^{e'} H^{q} \mathcal{E}'
					\cong
							\Omega_{Y}^{q - 1} / (1 + a C) \Omega_{Y, d = 0}^{q - 1}
						\oplus
							\Omega_{Y}^{q - 2} / (1 + a C) \Omega_{Y, d = 0}^{q - 1},
				\]
			where $a \in k$ is the residue class of $p \pi^{-e}$.
		\item
			If $m > e'$, then
				\[
						U^{m} H^{q} \mathcal{E}'
					=
						0.
				\]
	\end{enumerate}
\end{Prop}

\begin{proof}
	Apply $R \varepsilon_{\ast}$ to Proposition \ref{prop: calculation of graded pieces},
	use Proposition \ref{prop: Zariski stalks of nearby cycles}
	instead of \ref{prop: ind smooth approximation of cohomology}
	and argue as in \cite[Theorem (6.7)]{BK86}.
\end{proof}

For the proof of Theorem \ref{thm: main},
we may assume that the element $a$ above has a $(p - 1)$-st root in $k$.
Then $1 + a C$ in the statement may be replaced by $C - 1$.

Applying $R \varepsilon_{\ast}$ to \eqref{eq: duality morphism mod p with root of unity},
we have morphisms
	\[
			\mathcal{E}' \tensor^{L} \mathcal{E}'
		\to
			\mathcal{E}'
		\to
			R \varepsilon_{\ast} \nu(r)[- r - 1]
		\to
			\xi(r)[- r - 2].
	\]
The morphism $\mathcal{E}' \to \xi(r)[- r - 2]$ from the second term to the fourth term is also given by
	\[
			\mathcal{E}'
		\to
			H^{r + 2} \mathcal{E}'[-r - 2]
		=
			\gr^{e'} H^{r + 2} \mathcal{E}'[- r - 2]
		\to
			\xi(r)[- r - 2]
	\]
using Proposition \ref{prop: calculation of graded pieces in Zariski}.
For integers $s, t$ with $s + t = r + 2$, we thus have a pairing
	\begin{equation} \label{eq: Zariski duality pairing}
			H^{s} \mathcal{E}' \times H^{t} \mathcal{E}'
		\to
			H^{r + 2} \mathcal{E}'
		\to
			\xi(r).
	\end{equation}

\begin{Prop} \label{prop: Zariski duality}
	The pairing \eqref{eq: Zariski duality pairing} restricted to
	$U^{l} H^{s} \mathcal{E}' \times U^{m} H^{t} \mathcal{E}'$ is zero
	if $l + m > e'$
	and hence induces a pairing
	$\gr^{l} H^{s} \mathcal{E}' \times \gr^{m} H^{t} \mathcal{E}' \to \xi(r)$
	for $l + m = e'$.
	The induced morphism
		\[
				\gr^{l} H^{s} \mathcal{E}'
			\to
				R \sheafhom_{Y_{\RPSZ}}(\gr^{m} H^{t} \mathcal{E}', \xi(r))
		\]
	is an isomorphism if $l, m > 0$.
\end{Prop}

\begin{proof}
	The morphism $H^{s} \mathcal{E}' \times H^{t} \mathcal{E}' \to H^{r + 2} \mathcal{E}'$
	takes $U^{l} H^{s} \mathcal{E}' \times U^{m} H^{t} \mathcal{E}'$ to
	$U^{l + m} H^{r + 2} \mathcal{E}'$
	by Proposition \ref{prop: ind smooth approximation of cohomology}
	and \cite[Lemma (4.1)]{BK86}.
	The first assertion follows.
	For the second, consider the pairing
	between $\Omega_{Y}^{s - 1}$ and $\Omega_{Y}^{t - 1}$
	with values in $\Omega_{Y}^{r} / d \Omega_{Y}^{r - 1}$ given by the wedge product if $p \nmid m$
	and the pairing between $d \Omega_{Y}^{s - 1} \oplus d \Omega_{Y}^{s - 2}$ and
	$d \Omega_{Y}^{t - 1} \oplus d \Omega_{Y}^{t - 2}$
	with values in $\Omega_{Y}^{r} / d \Omega_{Y}^{r - 1}$ given by
		$
				(d \omega, d \omega') \times (d \tau, d \tau')
			\mapsto
				\omega \wedge d \tau' + \omega' \wedge d \tau
		$
	if $p \mid m$.
	Consider the composite of them with the natural surjection
	$\Omega^{r} / d \Omega_{Y}^{r - 1} \onto \xi(r)$.
	With the isomorphisms in Proposition \ref{prop: calculation of graded pieces in Zariski},
	we have a pairing
	$\gr^{l} H^{s} \mathcal{E}' \times \gr^{m} H^{t} \mathcal{E}' \to \xi(r)$
	for $l + m = e'$.
	This pairing agrees with the stated pairing up to an $\F_{p}^{\times}$-multiple
	if $l, m > 0$
	by the same argument as \cite[Lemma (5.2)]{BK86}.
	With this description and \cite[Theorem 3.2 (ii)]{Kat86},
	we see that the stated induced morphism is an isomorphism.
\end{proof}

\begin{Cor} \label{cor: Zariski duality, from U one to U e prime}
	Let $s + t = r + 2$.
	The pairing \eqref{eq: Zariski duality pairing} induces a pairing
	between $U^{1} H^{s} \mathcal{E}' / U^{e'} H^{s} \mathcal{E}'$
	and $U^{1} H^{t} \mathcal{E}' / U^{e'} H^{t} \mathcal{E}'$
	with values in $\xi(r)$.
	The induced morphism
		\[
				U^{1} H^{s} \mathcal{E}' / U^{e'} H^{s} \mathcal{E}'
			\to
				R \sheafhom_{Y_{\RPSZ}} \bigl(
					U^{1} H^{t} \mathcal{E}' / U^{e'} H^{t} \mathcal{E}',
					\xi(r)
				\bigr)
		\]
	is an isomorphism.
\end{Cor}

\begin{Prop} \label{prop: duality for U one}
	Let $s + t = r + 2$.
	The morphism
		\[
				U^{1} H^{s} \mathcal{E} \tensor^{L} U^{1} H^{t} \mathcal{E}
			\to
				\nu(r)[1]
		\]
	in $D(Y_{\RPS}, \Lambda)$ defined in
	Proposition \ref{prop: splitting pairing to unit part and gr zero part}
	gives a perfect duality.
\end{Prop}

\begin{proof}
	Let $M = U^{1} H^{s} \mathcal{E}' / U^{e'} H^{s} \mathcal{E}'$
	and $N = U^{1} H^{t} \mathcal{E}' / U^{e'} H^{t} \mathcal{E}'$.
	We have $\varepsilon^{\ast} U^{e'} H^{s} \mathcal{E}' = 0$
	and $\varepsilon^{\ast} U^{m} H^{s} \mathcal{E}' = U^{m} H^{s} \mathcal{E}$ for any $m$.
	Hence $\varepsilon^{\ast} M = U^{m} H^{s} \mathcal{E}$
	and $\varepsilon^{\ast} N = U^{m} H^{t} \mathcal{E}$.
	Since $M$ and $N$ are finite successive extensions of locally free sheaves on $Y$ of finite rank
	by Proposition \ref{prop: calculation of graded pieces in Zariski},
	we have $M \isomto R \varepsilon_{\ast} \varepsilon^{\ast} M$
	and $N \isomto R \varepsilon_{\ast} \varepsilon^{\ast} N$.
	The morphism $\varepsilon^{\ast} M \tensor^{L} \varepsilon^{\ast} N \to \nu(r)[1]$
	in question induces a morphism
	$M \tensor^{L} N \to R \varepsilon_{\ast} \nu(r)[1] \to \xi(r)$
	by adjunction,
	which agrees with the pairing in
	Corollary \ref{cor: Zariski duality, from U one to U e prime}.
	Hence the diagram
		\[
			\begin{CD}
					R \varepsilon_{\ast}
					\varepsilon^{\ast} M
				@>>>
					R \varepsilon_{\ast}
					R \sheafhom_{Y_{\RPS}}(\varepsilon^{\ast} N, \nu(r))[1]
				\\
				@|
				@|
				\\
					M
				@>>>
					R \sheafhom_{Y_{\RPSZ}}(N, \xi(r))
			\end{CD}
		\]
	is commutative,
	where the right vertical isomorphism is from
	Proposition \ref{prop: etale duality pushes to Zariski duality}.
	The lower horizontal morphism is an isomorphism
	by Corollary \ref{cor: Zariski duality, from U one to U e prime}.
	Hence the upper horizontal morphism is also an isomorphism.
	Its pullback 
		$
				\varepsilon^{\ast} M
			\to
				R \sheafhom_{Y_{\RPS}}(\varepsilon^{\ast} N, \nu(r))[1]
		$
	is thus an isomorphism.
	This gives the result.
\end{proof}

\begin{Prop} \label{prop: duality is true for n equals one}
	The statement of Theorem \ref{thm: main} \eqref{item: main thm, duality} is true for $n = 1$.
\end{Prop}

\begin{proof}
	This follows from
	Propositions \ref{prop: splitting pairing to unit part and gr zero part},
	\ref{prop: duality for gr zero} and
	\ref{prop: duality for U one}.
\end{proof}


\section{General case}
Let $n \ge 1$ and $s, t$ with $s + t = r + 1$.

\begin{Prop}
	The statement of Theorem \ref{thm: main} \eqref{item: main thm, cohom dimension} is true.
\end{Prop}

\begin{proof}
	The distinguished triangle
	$R \Psi \Lambda_{n - 1}(s) \to R \Psi \Lambda_{n}(s) \to R \Psi \Lambda_{1}(s)$
	and induction reduce the statement to the case $n = 1$
	already proven in Proposition \ref{prop: main thm, cohom dimension is true for n equal to one}.
\end{proof}

Hence we have a canonical morphism
	$
			R \Psi \Lambda_{n}(s)
		\tensor^{L}
			R \Psi \Lambda_{n}(t)
		\to
			\nu_{n}(r)[- r - 1]
	$
as explained in Theorem \ref{thm: main}.
We want to prove that it induces a perfect duality as stated.

Denote $R \sheafhom_{(\Lambda_{n})_{Y_{\RPS}}}$ by $[\;\cdot\;, \;\cdot\;]_{n}$.
For $m \le n$,
the exact inclusion $M(Y, \Lambda_{m}) \into M(Y, \Lambda_{n})$ induces a triangulated functor
$D(Y, \Lambda_{m}) \to D(Y, \Lambda_{n})$.
We denote it by $\theta$, but it is frequently omitted from the notation,
and the image of an object $M$ by $\theta$ is simply denoted by just $M$.
Set $\mathcal{E}_{n}^{s} = R \Psi \Lambda_{n}(s)[s]$.
Denote the morphism $\Tr \colon \mathcal{E}_{n}^{r + 1} \to \nu_{n}(r)$ by $\Tr_{n}$.
Hence we have canonical morphisms
	\[
			\mathcal{E}_{n}^{s} \tensor^{L} \mathcal{E}_{n}^{t}
		\to
			\mathcal{E}_{n}^{r + 1}
		\overset{\Tr_{n}}{\to}
			\nu_{n}(r).
	\]

\begin{Prop} \label{prop: duality morphism for smaller power of p}
	For $m \le n$,
	consider the composite of the morphisms
		\[
				\mathcal{E}_{m}^{s}
			\to
				[\mathcal{E}_{m}^{t}, \mathcal{E}_{m}^{r + 1}]_{m}
			\overset{\theta}{\to}
				[\mathcal{E}_{m}^{t}, \mathcal{E}_{m}^{r + 1}]_{n}
			\overset{\formal{p^{n - m}}}{\to}
				[\mathcal{E}_{m}^{t}, \mathcal{E}_{n}^{r + 1}]_{n}
			\overset{\Tr_{n}}{\to}
				[\mathcal{E}_{m}^{t}, \nu_{n}(r)]_{n}
		\]
	and the composite of the morphisms
		\[
				\mathcal{E}_{m}^{s}
			\to
				[\mathcal{E}_{m}^{t}, \mathcal{E}_{m}^{r + 1}]_{m}
			\overset{\Tr_{m}}{\to}
				[\mathcal{E}_{m}^{t}, \nu_{m}(r)]_{m}
			\overset{\theta}{\to}
				[\mathcal{E}_{m}^{t}, \nu_{m}(r)]_{n}
			\overset{\formal{p^{n - m}}}{\to}
				[\mathcal{E}_{m}^{t}, \nu_{n}(r)]_{n}.
		\]
	They are equal.
	If the statement of Theorem \ref{thm: main} \eqref{item: main thm, duality} is true for $m$,
	then they are isomorphisms.
\end{Prop}

\begin{proof}
	That they are equal follows from
	Proposition \ref{prop: trace morphism and formal multiplication by p}.
	The composite
		\[
				[\mathcal{E}_{m}^{t}, \nu_{m}(r)]_{m}
			\overset{\theta}{\to}
				[\mathcal{E}_{m}^{t}, \nu_{m}(r)]_{n}
			\overset{\formal{p^{n - m}}}{\to}
				[\mathcal{E}_{m}^{t}, \nu_{n}(r)]_{n}
		\]
	is an isomorphism by \cite[(4.2.4)]{Kat86}.
	Hence the second statement follows.
\end{proof}

\begin{Prop} \label{prop: two out of three for duality morphisms}
	The diagram
		\[
			\begin{CD}
					\mathcal{E}_{1}^{s}
				@>> \formal{p^{n - 1}} >
					\mathcal{E}_{n}^{s}
				@>>>
					\mathcal{E}_{n - 1}^{s}
				\\
				@VVV
				@VVV
				@VVV
				\\
					[\mathcal{E}_{1}^{t}, \nu_{n}(r)]_{n}
				@>>>
					[\mathcal{E}_{n}^{t}, \nu_{n}(r)]_{n}
				@> \formal{p} >>
					[\mathcal{E}_{n - 1}^{t}, \nu_{n}(r)]_{n}
			\end{CD}
		\]
	is a morphism of distinguished triangles,
	where the vertical morphisms are the morphisms of
	Proposition \ref{prop: duality morphism for smaller power of p}
	for $m = 1, n, n - 1$ from the left to the right.
\end{Prop}

\begin{proof}
	Applying $R \Psi[s]$ to the morphism of distinguished triangles
		\[
			\begin{CD}
					\Lambda_{1}(s)
				@>> \formal{p^{n - 1}} >
					\Lambda_{n}(s)
				@>>>
					\Lambda_{n - 1}(s)
				\\
				@VVV
				@VVV
				@VVV
				\\
					[\Lambda_{1}(t), \Lambda_{n}(r + 1)]_{n}
				@>>>
					[\Lambda_{n}(t), \Lambda_{n}(r + 1)]_{n}
				@> \formal{p} >>
					[\Lambda_{n - 1}(t), \Lambda_{n}(r + 1)]_{n},
			\end{CD}
		\]
	we have a morphism of distinguished triangles
		\[
			\begin{CD}
					\mathcal{E}_{1}^{s}
				@>> \formal{p^{n - 1}} >
					\mathcal{E}_{n}^{s}
				@>>>
					\mathcal{E}_{n - 1}^{s}
				\\
				@VVV
				@VVV
				@VVV
				\\
					[\mathcal{E}_{1}^{t}, \mathcal{E}_{n}^{r + 1}]_{n}
				@>>>
					[\mathcal{E}_{n}^{t}, \mathcal{E}_{n}^{r + 1}]_{n}
				@> \formal{p} >>
					[\mathcal{E}_{n - 1}^{t}, \mathcal{E}_{n}^{r + 1}]_{n}.
			\end{CD}
		\]
	The morphism $\Tr_{n}$ gives a morphism of distinguished triangles
	from the lower triangle of this diagram
	to the lower triangle of the stated diagram.
\end{proof}

\begin{Prop}
	The statement of Theorem \ref{thm: main} \eqref{item: main thm, duality} is true.
\end{Prop}

\begin{proof}
	This follows from Propositions \ref{prop: duality is true for n equals one},
	\ref{prop: duality morphism for smaller power of p}
	and \ref{prop: two out of three for duality morphisms}
	by induction.
\end{proof}


\end{document}